\documentclass[10pt]{amsart}
\usepackage{amsthm,amsmath,amssymb,amsbsy,amsfonts,enumerate,graphicx}
\usepackage[hmargin=1.5in,vmargin=1.5in]{geometry}

\newtheorem{prop}{Proposition}[section]
\newtheorem{coro}{Corollary}[section]
\newtheorem{lemma}{Lemma}[section]

\newcommand{\Q}{{\mathbb Q}}

\begin{document}

\title{Zeros and convergent subsequences of Stern polynomials}
\author{Antonio R. Vargas}
\address{Dept. of Mathematics and Statistics, Dalhousie University, Halifax, Nova Scotia B3H 1Z9, Canada}

\begin{abstract}
We investigate Dilcher and Stolarsky's polynomial analogue of the Stern diatomic sequence.  Basic information is obtained concerning the distribution of their zeros in the plane.  Also, uncountably many subsequences are found which each converge to a unique analytic function on the open unit disk.  We thus generalize a result of Dilcher and Stolarsky from their second paper on the subject.
\end{abstract}

\maketitle

\section{Introduction}

The Stern sequence $\{a(n)\}_{n \geq 0}$, also known as the Stern diatomic sequence (or Stern's diatomic series), is an integer sequence defined recursively by $a(0) = 0$, $a(1) = 1$, and
\begin{align*}
	&a(2n) = a(n), \\
	&a(2n+1) = a(n) + a(n+1).
\end{align*}
Dilcher and Stolarsky recently introduced a polynomial analogue of this sequence and explored some of its properties \cite{dilcher:sternpolys01,dilcher:sternpolys02} (see \cite{dilcher:sternpolys01} for background information and references for the classical Stern sequence).  This new sequence $\{a(n;z)\}_{n \geq 0}$ is defined by $a(0;z) = 0$, $a(1;z) = 1$, and
\begin{equation}
	\begin{split}
	&a(2n;z) = a(n;z^2), \\
	&a(2n+1;z) = z\,a(n;z^2) + a(n+1;z^2).
	\end{split}
	\label{spdef}
\end{equation}
These polynomials have only $0$ and $1$ as coefficients, and $a(n;1)=a(n)$.  See Table \ref{seqlist} for a list of the first few polynomials in the sequence.

As nothing has yet been said about the zeros of these polynomials, we discuss the basics of their distribution in the plane in Section \ref{clustcirc}.

The heart of this paper lies in Section \ref{insidecircle}, where we switch to a more structural point of view to show that subsequences of \eqref{spdef} of a certain form converge to analytic functions on the open unit disk.  In fact, there are uncountably many such limit functions.  This generalizes a result of Dilcher and Stolarsky in \cite{dilcher:sternpolys02} where two particular convergent subsequences of \eqref{spdef} were found.  The transcendence of the limit functions of these two subsequences has been of interest: Coons \cite{coons:transcendence} showed that the limit functions are transcendental over $\Q(z)$, the field of rational functions over $\Q$, and Adamczewski \cite{adam:contfracs} further showed that the functions take transcendental values at every algebraic number in the punctured unit disk.  Also, Fibonacci-like quadratic relations were derived in \cite{dilcher:sternpolys02} for the elements of these two subsequences---an effort continued by Vsemirnov in \cite{vsemirnov:sternpolys}.  We will find these particular subsequences and functions among those introduced in this paper in Section \ref{dilstoseqs}.

We take a moment here to note that a different polynomial analogue to the Stern sequence was given by Klav{\v{z}}ar et. al. in \cite{petr:altsternpolys}.  A similar sequence to this was defined by Alkauskas \cite{alkauskas:likealtsternpolys} in relation to p-continued fractions.

\section{Zeros}
\label{clustcirc}

In this section we show that the zeros of the Stern polynomials cluster uniformly near the unit circle in the sense of Weyl.  That is, we can prove the following result.

\begin{prop}
Let $\sharp_n^\circ(\rho)$ be the number of zeros of $a(n;z)$ in the annulus \linebreak $1-\rho \leq |z| \leq 1/(1-\rho)$, and let $\sharp_n^\angle(\theta_1,\theta_2)$ be the number of zeros of $a(n;z)$ in the sector $\theta_1 \leq \arg z < \theta_2$.  For fixed $\rho$, $\theta_1$, and $\theta_2$ satisfying $0 < \rho < 1$ and $0 \leq \theta_1 < \theta_2 \leq 2\pi$,
\[
	-16 \sqrt{\frac{3 \log n - \log 8}{n}} < \frac{\sharp_n^\angle(\theta_1,\theta_2)}{\deg a(n;z)} - \frac{\theta_2-\theta_1}{2\pi} < 16 \sqrt{\frac{3 \log n - \log 8}{n}}
\]
and
\[
	0 \leq 1 - \frac{\sharp_n^\circ(\rho)}{\deg a(n;z)} < \frac{2}{\rho} \cdot \frac{3\log n - \log 8}{n}
\]
for all $n > 0$ which are not powers of 2.
\label{unitcircprop}
\end{prop}

To prove this we require three results, the first due to Dilcher and Stolarsky \cite{dilcher:sternpolys01}, the second due to Erd\H{o}s and Tur{\'a}n \cite{erdos:zeroargdistr}, and the third due to Hughes and Nikeghbali \cite{hughes:zerosrandpolys}.

\begin{prop}[Dilcher and Stolarsky]
	Let $e(n) = \textnormal{ord}_2 (n)$, the highest power of $2$ dividing $n$.  Then for $n \geq 1$,
	\[
		\deg a(n;z) = \frac{n - 2^{e(n)}}{2}.
	\]
	\label{dilstothm}
\end{prop}

\begin{prop}[Erd\H{o}s and Tur{\'a}n]
For a polynomial
\[
	P(z) = \sum_{j=0}^{N} \alpha_j z^j
\]
with $\alpha_0 \alpha_N \neq 0$, define
\[
	L(P) = \log \sum_{j=0}^{N} |\alpha_j| - \frac{1}{2} \log |\alpha_0| - \frac{1}{2} \log |\alpha_N|.
\]
If $\sharp_P^\angle(\theta_1,\theta_2)$ is the number of zeros of $P(z)$ in the sector $\theta_1 \leq \arg z < \theta_2$, then, for \linebreak $0 \leq \theta_1 < \theta_2 \leq 2\pi$,
\[
	\left(\frac{\theta_2-\theta_1}{2\pi} - \frac{\sharp_P^\angle(\theta_1,\theta_2)}{N}\right)^2 < 256 \,\frac{L(P)}{N}.
\]
\label{erdprop}
\end{prop}

\begin{prop}[Hughes and Nikeghbali]
For a polynomial
\[
	P(z) = \sum_{j=0}^{N} \alpha_j z^j
\]
with $\alpha_0 \alpha_N \neq 0$, let $L(P)$ be as defined in Proposition \ref{erdprop}.  If $\sharp_P^\circ(\rho)$ is the number of zeros of $P(z)$ in the annulus $1-\rho \leq |z| \leq 1/(1-\rho)$, then, for $0 < \rho < 1$,
\[
	1 - \frac{\sharp_P^\circ(\rho)}{N} \leq \frac{2}{\rho} \cdot \frac{L(P)}{N}.
\]
\label{hughesprop}
\end{prop}

We may now prove Proposition \ref{unitcircprop}.

\begin{proof}[Proof of Proposition \ref{unitcircprop}]
In light of Propositions \ref{erdprop} and \ref{hughesprop}, the results will follow from estimating the quantity $L(a(n;z))/\deg a(n;z)$.

Let $e(n)$ be the largest power of $2$ dividing $n$, and let $d(n) = n/2^{e(n)}$ be the odd component of $n$.  In these terms, $d(n) = 1$ is necessary and sufficient for $a(n;z) = 1$, so we will only consider $n$ such that $d(n) \geq 3$.  With the aid of Proposition \ref{dilstothm} we calculate
\begin{align*}
	L({a(n;z)}) &\leq \log \deg a(n;z) \\
							&< \log n - \log 2
\end{align*}
and
\[
	\deg\, a(n;z) \geq \frac{n}{3}
\]
to get
\[
	\frac{L(a(n;z))}{\deg\, a(n;z)} < \frac{3 \log n - \log 8}{n}.
\]
This completes the proof.
\end{proof}

We conclude this section with a small remark about real zeros.

It follows from Proposition \ref{dilstothm} that the only Stern polynomials of odd degree are those of the form $a(4n+3;z)$.  Now, $a(4n+3;-1) = -a(n) < 0$ and $a(4n+3;0) = 1$, so $a(4n+3;x)$ (where $x$ is real) has a zero in the interval $(-1,0)$.  Further, using \eqref{spdef} it is not difficult to show that on $(-\infty,-1)$ we have $a(4n+3;x) < 0$, on $(0,\infty)$ we have $a(4n+3;x) > 0$, and on $(-1,0)$ we calculate
\[
	\frac{d}{dx} a(4n+3;x) = (3x^2+4x^6)\,a'(n;x^4) + (1+4x^3+4x^4)\,a'(n+1;x^4) > 0.
\]
Hence this real zero of $a(4n+3;z)$ is unique.

\section{Convergent subsequences}
\label{insidecircle}

Given the definition of this polynomial sequence it is natural to look at subsequences of the form $\{a(2^{m+1} n + \ell;z)\}_{n \geq 0}$ for fixed integers $m \geq 0$ and $0 \leq \ell < 2^{m+1}$.  To facilitate this we will enlist the help of binary sequences (that is, sequences whose elements are only $0$'s and $1$'s).  We will denote a finite sequence by $\{b_k\}_{0 \leq k \leq m}$ and an infinite sequence by $\{b_k\}$.

Suppose we have a binary sequence $\{b_k\}_{0 \leq k \leq m}$.  With this sequence we may associate the polynomial sequence
\[
	\left\{\!a\!\!\left(\!2^{m+1} n +\!\sum_{j = 0}^{m} 2^j b_j\,;\,z\!\right)\!\right\}_{n \geq 0}\!.
\]
For fixed $m$, every polynomial sequence of the form $\{a(2^{m+1} n + \ell;z)\}_{n \geq 0}$ is associated with exactly one finite binary sequence.

But we are more interested in what occurs when we instead leave $n$ fixed, begin with an infinite binary sequence $\{b_k\}$, and allow $m$ to vary, producing the sequences
\[
	\left\{a\!\!\left(\!2^{m+1}n +\!\sum_{j = 0}^{m} 2^j b_j\,;\,z\!\right)\right\}_{m \geq 0}.
\]
These sequences are unique in that as $m$ increases the heads of the polynomials stabilize, so that they gradually build a formal Maclaurin series.  The series built turns out to be independent of $n$, thus we identify it with the binary sequence $\{b_k\}$.

\begin{prop}
Given a binary sequence $\{b_k\}$, the sequences of polynomials
\[
	\left\{a\!\!\left(\!2^{m+1}n +\!\sum_{j = 0}^{m} 2^j b_j\,;\,z\!\right)\right\}_{m \geq 0} \text{ for all fixed $n > 0$}
\]
converge to the same analytic function $f_{\{b_k\}}$ on the open unit disk.
\label{analfuncprop}
\end{prop}

It will be shown in the proof of Proposition \ref{analfuncprop} that the functions $f_{\{b_k\}}$ are never polynomials, so that the elements of the Stern polynomial sequence \eqref{spdef} can be thought of as partial sums of appropriate power series with radius of convergence $1$.  In this sense, the clustering behavior of the zeros of the Stern polynomials described in Proposition \ref{unitcircprop} is explained by the classical theorem of Jentzsch \cite{jentzsch}, which states that every point on the circle of convergence of a power series is a limit point of the zeros of its partial sums.

The first few terms of one of the subsequences from Proposition \ref{analfuncprop} are shown in Table \ref{convlist}.  In the context of Hurwitz's Theorem (see e.g. \cite[p. 4]{marden:geometry}), the result is illustrated graphically in Figure \ref{clustconv}, which displays the convergence of the zeros of the polynomials.
\begin{table}[htb]
	\centering
	\begin{tabular}{r|l}
		$m$ & $a\!\left(3\cdot2^{m+1} +\!\sum_{j=0}^m 2^j b_j;z\right)$ \\ \hline
		$0$ & $1 + z^2$ \\
		$1$ & $1 + z^2 + z^6$ \\
		$2$ & $1+z^2+z^4+z^{10}+z^{12}$ \\
		$3$ & $1+z^2+z^4+z^8+z^{18}+z^{20}+z^{24}$ \\
		$4$ & $1+z^2+z^4+z^8+z^{18}+z^{20}+z^{24}+z^{50}+z^{52}+z^{56}$ \\
		$5$ & $1+z^2+z^4+z^8+z^{18}+z^{20}+z^{24}+z^{32}+z^{34}+z^{36}+z^{40}+z^{82}+z^{84}$ \\
		    & \hspace{0.14cm} $+z^{88}+z^{96}+z^{98}+z^{100}+z^{104}$ \\
		$6$ & $1+z^2+z^4+z^8+z^{18}+z^{20}+z^{24}+z^{32}+z^{34}+z^{36}+z^{40}+z^{64}+z^{66}$ \\
		    & \hspace{0.14cm} $+z^{68}+z^{72}+z^{146}+z^{148}+z^{152}+z^{160}+z^{162}+z^{164}+z^{168}+z^{192}$ \\
		    & \hspace{0.14cm} $+z^{194}+z^{196}+z^{200}$
	\end{tabular}
	\vspace{2mm}
	\caption{The first few elements in the sequence from Proposition \ref{analfuncprop} with $n = 3$ and $\{b_k\} = \{0,1,0,0,1,0,0\}$.}
	\label{convlist}
\end{table}

\begin{figure}[h!tb]
	\centering
	\begin{tabular}{ccc}
		\includegraphics[width=0.3\textwidth]{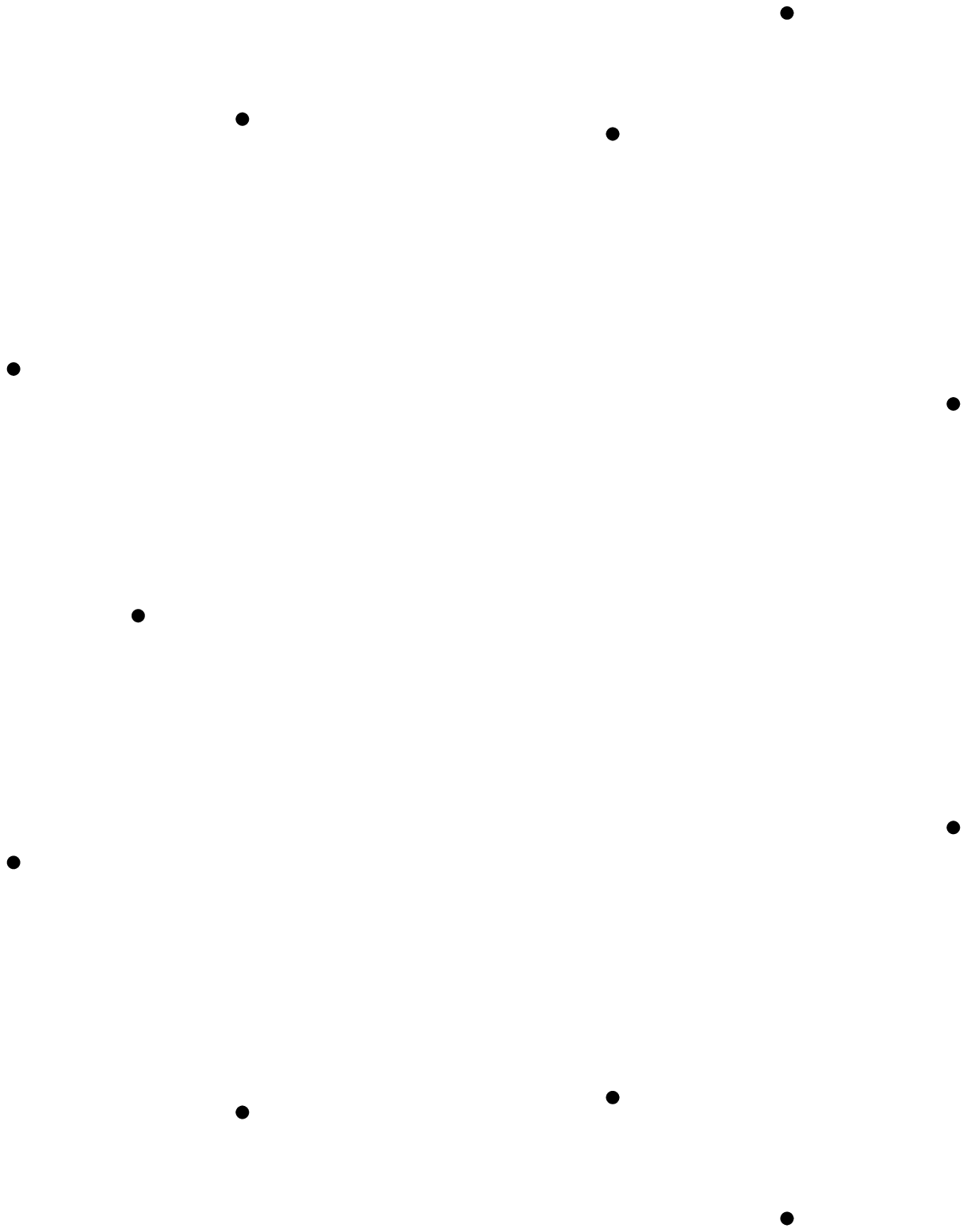}
			& \includegraphics[width=0.3\textwidth]{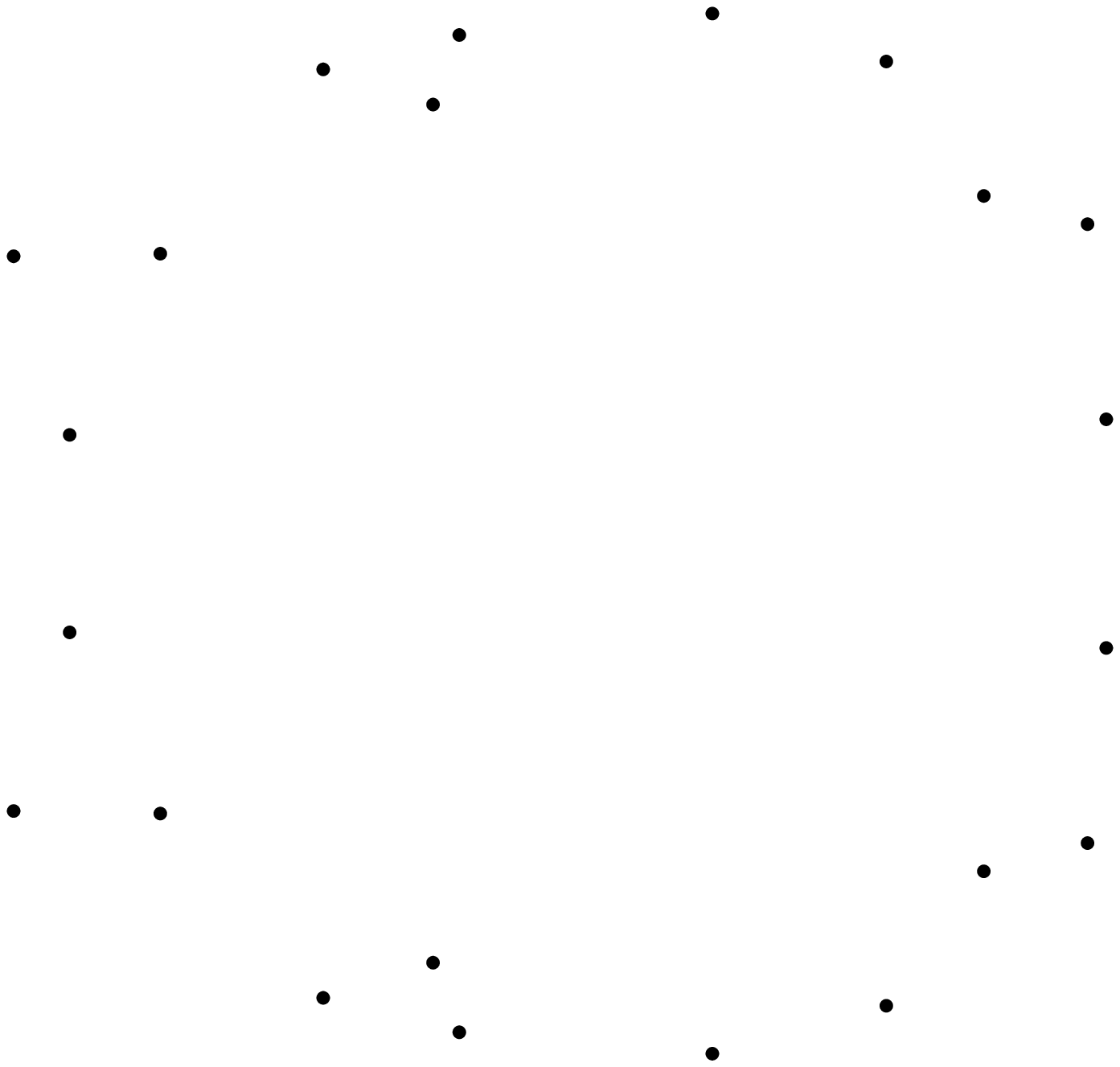}
			& \includegraphics[width=0.3\textwidth]{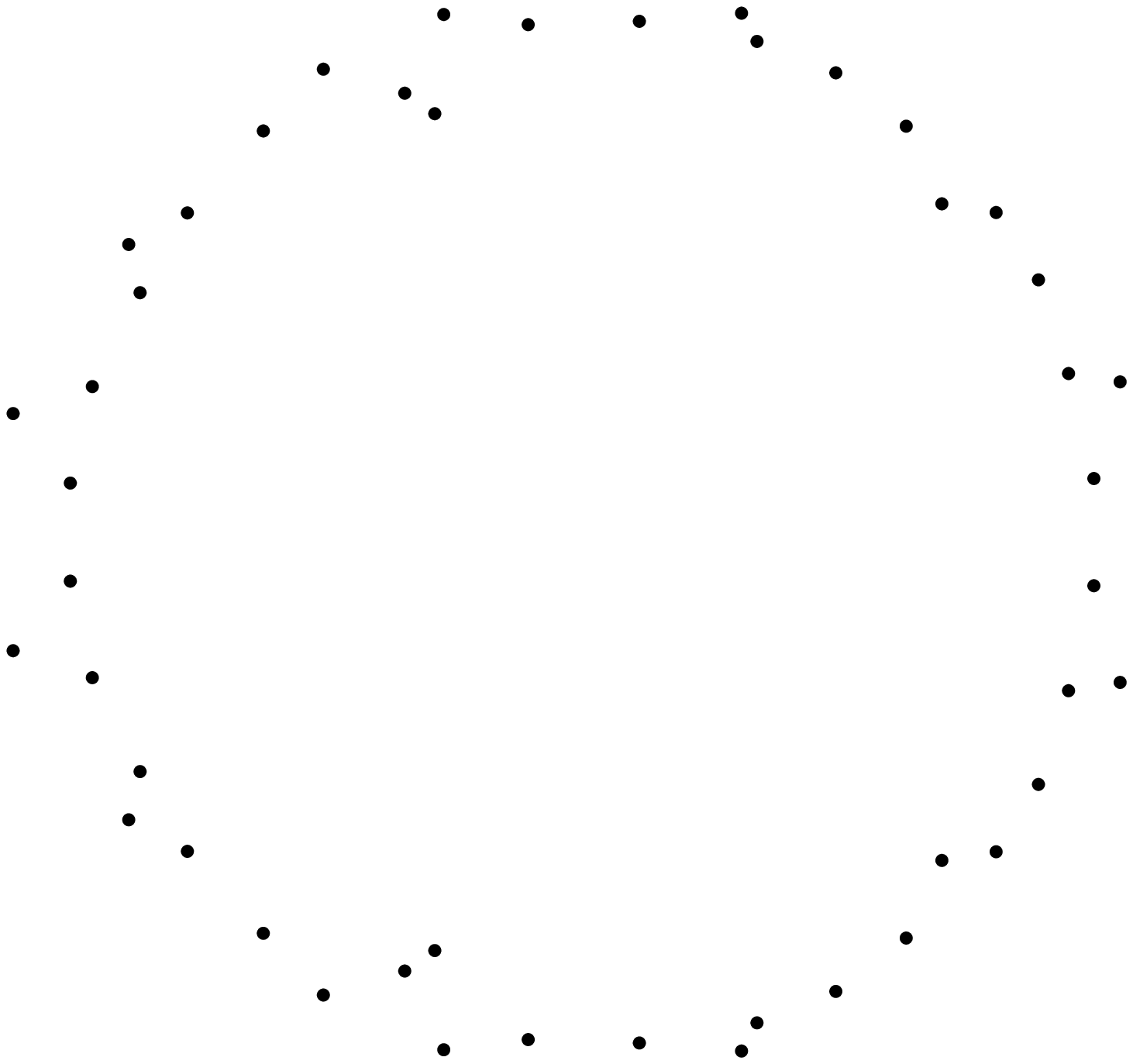} \\
		\includegraphics[width=0.3\textwidth]{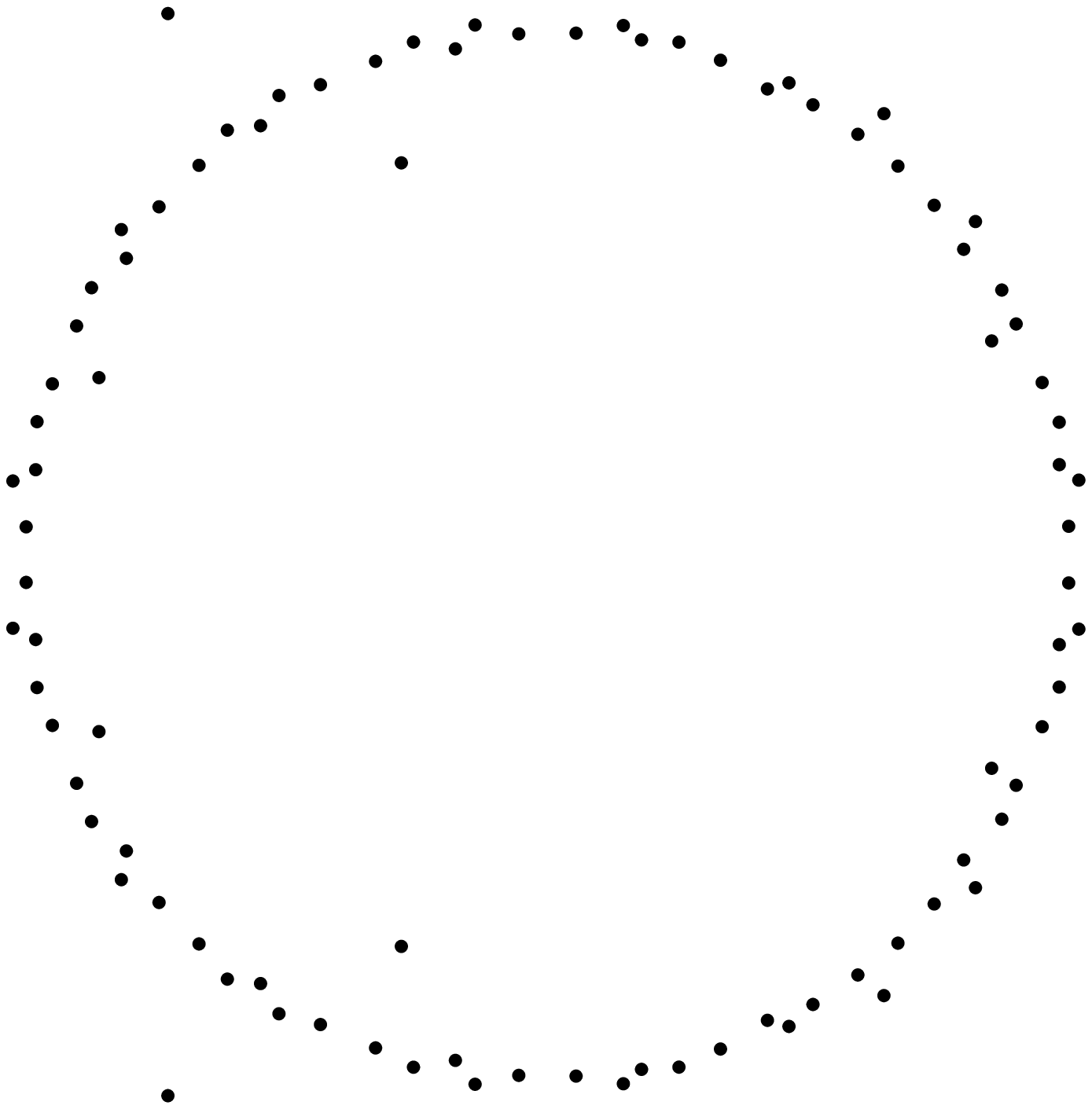}
			& \includegraphics[width=0.3\textwidth]{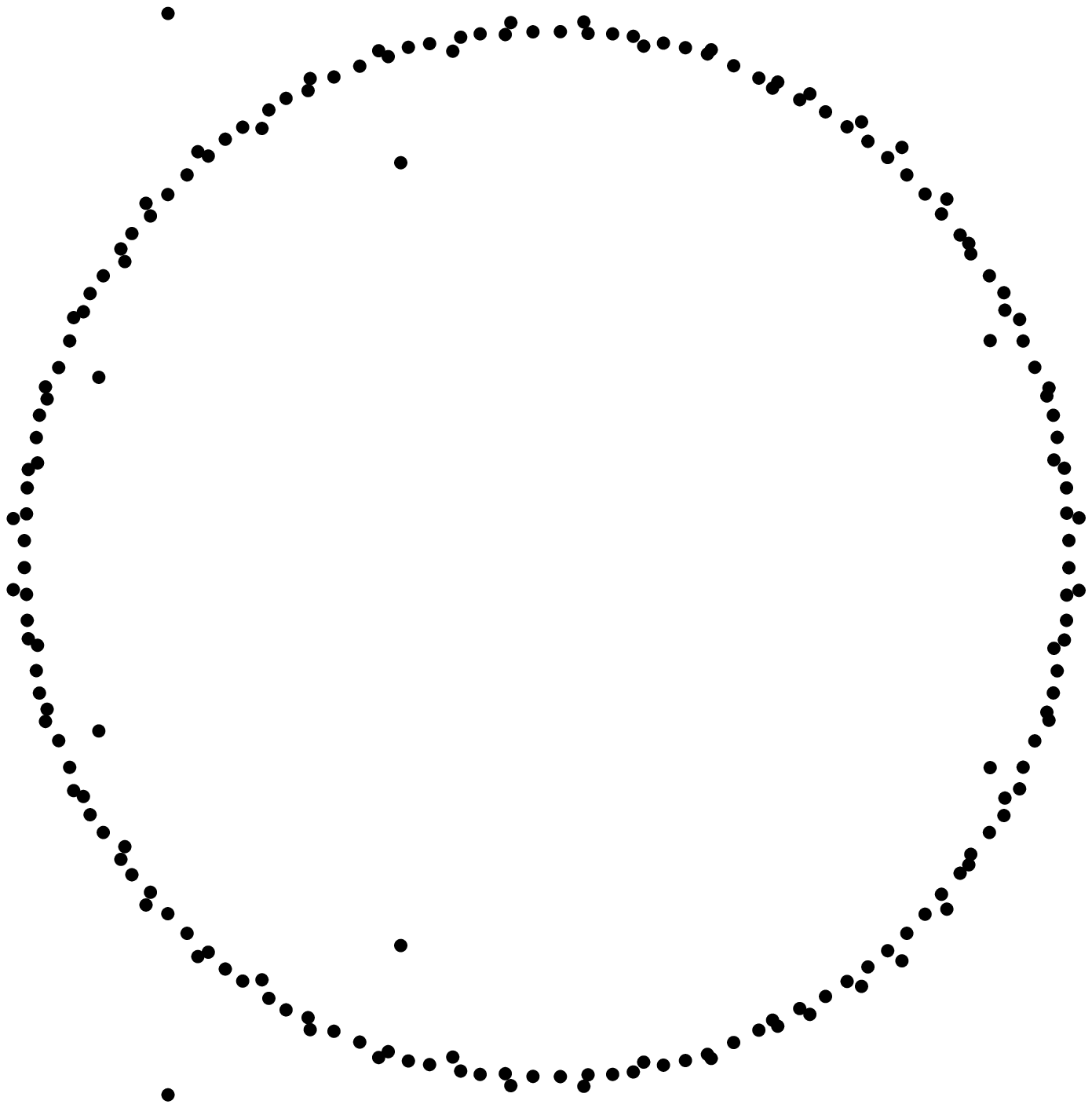}
			& \includegraphics[width=0.3\textwidth]{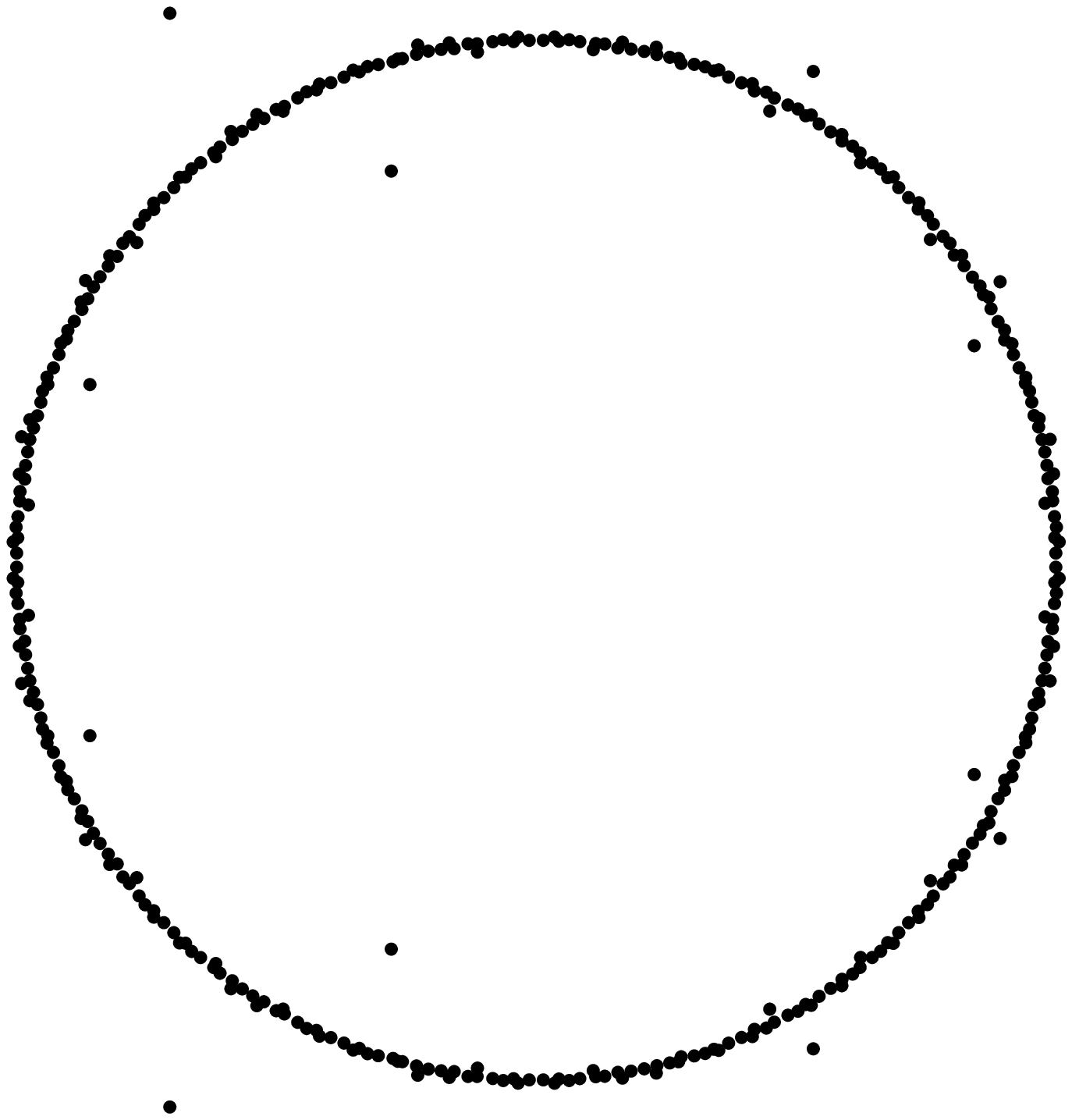} \\
		\includegraphics[width=0.3\textwidth]{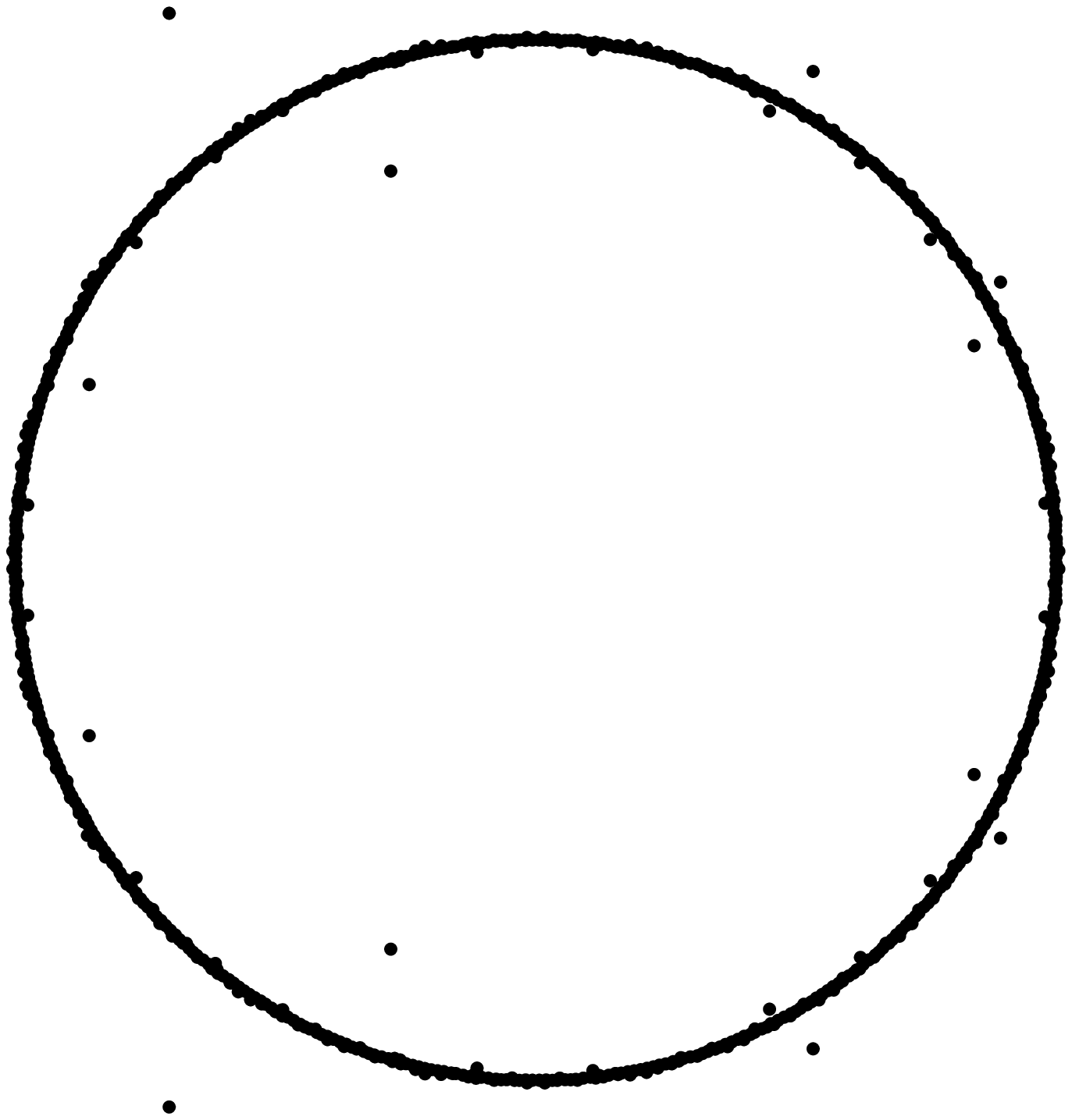}
			& \includegraphics[width=0.3\textwidth]{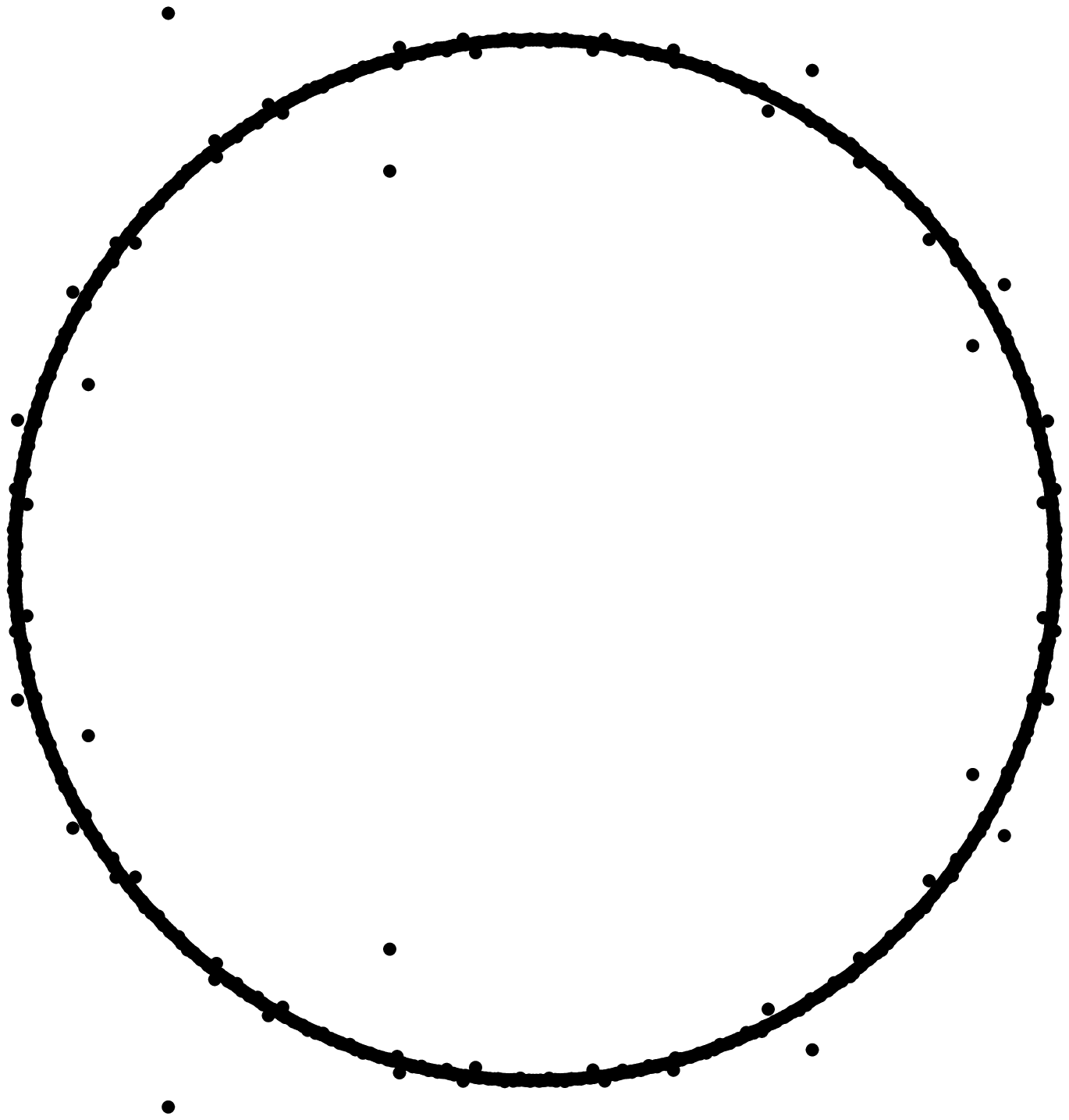}
			& \includegraphics[width=0.3\textwidth]{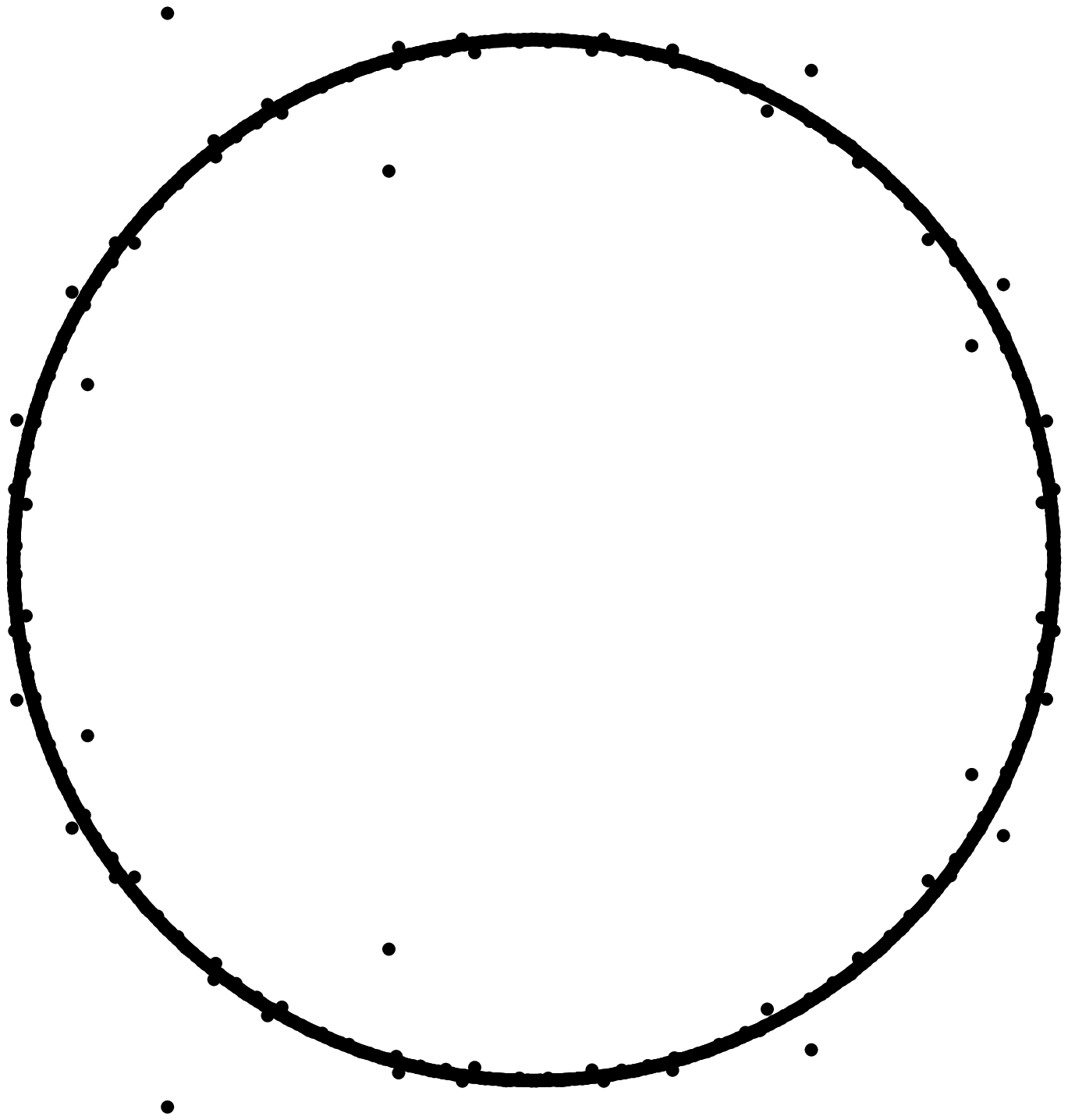}
	\end{tabular}
	\caption{Nine plots containing zeros of elements of the polynomial sequence from Proposition \ref{analfuncprop} with $\{b_k\} = \{1,0,1,0,1,0,1,0,1\}$ and $n = 11$.  The index $m$ increases from $0$ in the top left to $8$ in the bottom right, with $11$, $22$, $46$, $90$, $186$, $362$, $746$, $1450$, and $2986$ zeros in each plot, respectively.}
	\label{clustconv}
\end{figure}

The foundation of the proof of Proposition \ref{analfuncprop} is laid by the following lemma, whose proof is a straightforward exercise in mathematical induction.

\begin{lemma}
	Let $\{b_k\}$ be a binary sequence.  Define the two polynomial sequences $\{s_k\}$ and $\{t_k\}$ by $s_{-1} = 1$, $t_{-1} = 0$, and
	\begin{align*}
		&s_k = (1 - b_k) s_{k-1} + z^{2^k}\!\left[b_k(s_{k-1} - t_{k-1}) + t_{k-1}\right]\!, \\
		&t_k = b_k s_{k-1} + t_{k-1}.
	\end{align*}
	Then
	\[
		a\!\!\left(\!2^{m+1} n +\!\sum_{j = 0}^{m} 2^j b_j\,;\,z\!\right) =\,s_m\,a\!\left(\!n;z^{2^{m+1}}\!\right) + t_m\,a\!\left(\!n+1;z^{2^{m+1}}\!\right)\!.
	\]
	\label{smtm}
\end{lemma}

Rearranging the formula in Lemma \ref{smtm} we get
\begin{equation}
	a\!\!\left(\!2^{m+1} n +\!\sum_{j = 0}^{m} 2^j b_j\,;\,z\!\right) = A_{m,n} + (s_m + t_m)\,a\!\left(\!n+1;z^{2^{m+1}}\!\right)\!,
	\label{smtmrearr1}
\end{equation}
where
\begin{equation}
	A_{m,n} = \,s_m\!\left[\!a\!\left(\!n;z^{2^{m+1}}\!\right) - a\!\left(\!n+1;z^{2^{m+1}}\!\right)\!\right]\!.
	\label{smtmrearr2}
\end{equation}
The following calculation will be useful.

\begin{lemma}
	The degree of the lowest power of $z$ in $A_{m,n}$ is at least $2^{m+1} + \sum_{j=0}^{m} 2^{j} b_j$ when $n > 1$.
	\label{degbd}
\end{lemma}

\begin{proof}
Suppose that $1$ occurs at least once in the sequence $\{b_k\}_{0 \leq k \leq m}$ and let $k_0$ be the index of its first occurrence.  We then have $s_{k_0-1} = 1$ and $t_{k_0-1} = 0$, so that $s_{k_0} = z^{2^{k_0}}$\!.  If $1$ occurs at least twice in $\{b_k\}_{0 \leq k \leq m}$, let $k_1$ be the index of its second occurrence.  Then
\[
	s_{k_1-1}\,= \sum_{j=0}^{k_1-k_0-1} z^{2^{k_0+j}},
\]
which tells us that the degree of the lowest power of $z$ in $s_{k_1}$ is $2^{k_0} + 2^{k_1}$.  Continuing in this way we see that, if $\{k_0, k_1, \ldots, k_p\}$ is the complete set of indices of occurrences of 1 in $\{b_k\}_{0 \leq k \leq m}$, the degree of the lowest power of $z$ in $s_m$ is
\[
	\sum_{j=0}^{p} 2^{k_j}\,=\,\sum_{j=0}^{m} 2^{j} b_j.
\]
The result follows.
\end{proof}

This, combined with Proposition \ref{dilstothm}, will allow us to prove Proposition \ref{analfuncprop}.

\begin{proof}[Proof of Proposition \ref{analfuncprop}]
We will first prove the result for $n = 1$.  If $1$ never occurs in the sequence $\{b_k\}$, then the ``polynomials'' in question are just the constant function with value $1$, so suppose it occurs at least once.  In the notation of Proposition \ref{dilstothm} we have
\[
	e\!\!\left(\!2^{m+1} +\!\sum_{j=0}^{m} 2^j b_j\!\right) =\,k_0,
\]
where $k_0$ is the index of the first occurrence of $1$ in $\{b_k\}$.  Proposition \ref{dilstothm} then tells us that
\begin{equation}
	\deg(s_m + t_m) \,=\, \deg\, a\!\!\left(\!2^{m+1} +\!\sum_{j = 0}^{m} 2^j b_j\,;\,z\!\right) =\, 2^m +\!\sum_{j = k_0 + 1}^{m}\! 2^{j-1} b_j.
	\label{smtmdeg}
\end{equation}

Note that
\[
	s_m + t_m = s_{m-1} + t_{m-1} + z^{2^m}\!\left[b_m (s_{m-1} - t_{m-1}) + t_{m-1}\right]\!.
\]
From \eqref{smtmdeg} we have that the degree of $s_{m-1} + t_{m-1}$ is strictly less than $2^m$, so that the limit of $s_m + t_m$ as $m \to \infty$ is a formal power series $\sum \alpha_j z^j$ with $\alpha_0 = 1$ and $\alpha_j = 0$ or $\alpha_j = 1$ for all $j > 0$.  By comparison with the geometric series $\sum z^j$ this series converges on the open unit disk to an analytic function $f_{\{b_k\}}$.

Now, dropping any restrictions on the sequence $\{b_k\}$, we will consider the case where $n > 1$.  Equation \eqref{smtmdeg} implies that
\begin{equation*}
	\deg\!\left(s_m + t_m\right) <\, 2^{m+1}.
\end{equation*}
According to Lemma \ref{degbd}, the value on the right-hand side of the above inequality is at most the degree of the lowest power of $z$ in the polynomial $A_{m,n}$ defined in \eqref{smtmrearr1} and \eqref{smtmrearr2}.  Since the constant term of $a(n+1;z^{2^{m+1}})$ is $1$ we can now write
\[
	a\!\!\left(\!2^{m+1} n +\!\sum_{j = 0}^{m} 2^j b_j\,;\,z\!\right) =\,s_m + t_m + \text{higher order terms,}
\]
where the higher order terms are $O\!\left(\!z^{2^{m+1}}\!\right)\!$ as $z \to 0$.  Hence, for any $n > 0$,
\[
	\lim_{m \to \infty} a\!\!\left(\!2^{m+1} n +\!\sum_{j = 0}^{m} 2^j b_j\,;\,z\!\right) =\, \lim_{m \to \infty} \!\left(s_m + t_m\right) =\, f_{\{b_k\}}
\]
on the open unit disk.
\end{proof}

It follows from the above proof and Lemma \ref{smtm} that no two distict binary sequences may produce the same power series.

We remark that the zeros of the polynomials seem to converge outside of the unit circle as well, but that the series in question diverge there.  Carlitz's study \cite{carlitz:bell} of the reciprocals of the Stern polynomials may be of some use in exploring this phenomenon.  We also note that it is not clear whether the functions $f_{\{b_k\}}$ ever have the unit circle as a natural boundary.  The method used in the proof of Proposition \ref{analfuncprop} does not give enough information about the possible degrees of consecutive terms in the power series to apply, for example, the Ostrowski-Hadamard gap theorem.

\section{The subsequences of Dilcher and Stolarsky}
\label{dilstoseqs}

Define the integer sequence $\{\xi_n\}$ by
\[
	\xi_n = \frac{1}{3}\left(2^n - (-1)^n \right)\!,
\]
and the polynomials $\{\phi_{n}\}$ by
\[
	\phi_n(q) = a(\xi_n ; q).
\]
These polynomials were studied by Dilcher and Stolarsky in \cite{dilcher:sternpolys02} (we continue their use of $q$ as the variable).  In particular, they defined the limit functions
\begin{align*}
	&F(q) = \lim_{n \to \infty} \phi_{2n}(q), \\
	&G(q) = \lim_{n \to \infty} \phi_{2n+1}(q).
\end{align*}
Let us find these sequences among those considered in the present paper.

Let $\{\beta_n\}$ be the binary sequence $\{1,0,1,0,\ldots\}$.  Then
\begin{align*}
	2^{2n-2} +\!\sum_{j = 0}^{2n-3} 2^j \beta_j \,&=\, \sum_{j = 0}^{n-1} 4^j \\
																								&=\, \frac{1}{3}\left(4^n - 1\right) \\
																								&=\, \xi_{2n},
\end{align*}
so that
\[
	s_{2n-3} + t_{2n-3} = \phi_{2n},
\]
where $s_n$ and $t_n$ are defined as in Lemma \ref{smtm} for the binary sequence $\{\beta_n\}$.  From this we see that $F = f_{\{\beta_n\}}$.

Let $\{\gamma_n\}$ be the binary sequence defined by $\gamma_0 = 1$ and $\gamma_n = \beta_{n-1}$ for $n \geq 1$, so that $\{\gamma_n\} = \{1,1,0,1,0,\ldots\}$.  We proceed similarly, finding
\begin{align*}
	2^{2n-1} +\!\sum_{j = 0}^{2n-2} 2^j \gamma_j \,&=\, 1 + \frac{1}{2}\sum_{j = 1}^{n} 4^j \\
																								 &=\, \frac{1}{3}\left(2^{2n+1} + 1\right) \\
																								 &=\, \xi_{2n+1},
\end{align*}
so that
\[
	s_{2n-2} + t_{2n-2} = \phi_{2n+1},
\]
where $s_n$ and $t_n$ are defined as in Lemma \ref{smtm} for the binary sequence $\{\gamma_n\}$.  Thus we have $G = f_{\{\gamma_n\}}$.

The study of these subsequences was motivated by the fact that $\phi_{n}(1) = F_{n}$, the $n^\text{th}$ Fibonacci number \cite{lehmer:stern}.  Inspired by the recurrence $F_{n} = F_{n-1} + F_{n-2}$, Dilcher and Stolarsky derived the relations
\begin{equation*}
	\begin{split}
	&\phi_{2n}(q) = \phi_{2n-1}(q^2) + q\,\phi_{2n-2}(q^4) \\
	&\text{\hspace{22mm} and} \\
	&\phi_{2n+1}(q) = q\,\phi_{2n}(q^2) + \phi_{2n-1}(q^4),
	\end{split}
\end{equation*}
and, drawing from the identity $F_{n-1} F_{n+1} - F_{n}^{2} = (-1)^{n}$,
\begin{equation*}
	\begin{split}
	&\phi_{2n+1}(q) \phi_{2n-1}(q^2) - q\,\phi_{2n}(q) \phi_{2n}(q^2) = 1 \\
	&\text{\hspace{29.7mm} and} \\
	&\phi_{2n+1}(q) \phi_{2n+1}(q^2) - q\,\phi_{2n+2}(q) \phi_{2n}(q^2) = 1.
	\end{split}
\end{equation*}
Taking the limit as $n$ goes to infinity immediately yields interesting functional relations between $F$ and $G$, but perhaps it is more crucial that $F$ and $G$ satisfy
\begin{equation*}
	\begin{split}
	&F(q) G(q^{2^{2n}}) - q^{\xi_{2n}} G(q) F(q^{2^{2n}}) = \phi_{2n}(q), \\
	&G(q) G(q^{2^{2n-1}}) - q^{\xi_{2n-1}} F(q) F(q^{2^{2n-1}}) = \phi_{2n-1}(q),
	\end{split}
\end{equation*}
for all $n \geq 1$ \cite{dilcher:sternpolys02,vsemirnov:sternpolys}.  The sequence $\{\phi_{n}\}$ is unique in its relation to the Fibonacci sequence, but we suspect that $\{\phi_{2n}\}$ and $\{\phi_{2n+1}\}$ are are not the only subsequences introduced in Proposition \ref{analfuncprop} which satisfy these types of identities.

For example, consider the nonzero binary sequence with the simplest structure, $\{\delta_n\} = \{1,1,1,\ldots\}$.  If we define the integer sequence $\{\nu_n\}$ by
\[
	\nu_n = 2^{n-1} + \sum_{j=0}^{n-2} 2^j \delta_n,
\]
and the polynomial sequence $\{\psi_n\}$ by
\[
	\psi_n(q) = a(\nu_n;q),
\]
then we have
\begin{equation}
	\nu_n = 2 \nu_{n-1} + 1 = \nu_{n-1} + 2^{n-1}
	\label{index}
\end{equation}
and, from \eqref{spdef},
\begin{equation}
	\psi_n(q) = q\psi_{n-1}(q^2) + 1.
	\label{recurrence}
\end{equation}
Induction will show that
\begin{equation}
	\psi_n(q) + q^{\delta_n} = \psi_{n+1}(q),
	\label{recurrence2}
\end{equation}
which, when combined with \eqref{recurrence}, yields
\[
	\psi_n(q) \psi_n(q^2) - \psi_{n+1}(q) \psi_{n-1}(q^2) = q^{2\delta_{n-1}}.
\]
Unfortunately, no functional relation can be obtained from this by letting $n \to \infty$.  On the other hand, we also have from \eqref{recurrence} the identity
\begin{equation}
	\psi_n(q) - q^3\psi_{n-2}(q^4) = 1 + q = \psi_2(q),
	\label{basecase}	
\end{equation}
and if we let $n \to \infty$ we deduce that
\begin{equation}
	f_{\{\delta_n\}}(q) - q^3f_{\{\delta_n\}}(q^4) = \psi_2(q).
	\label{functional}	
\end{equation}
Equation \eqref{basecase}, combined with a simple induction argument using \eqref{index}, \eqref{recurrence}, and \eqref{recurrence2}, yields the following result.

\begin{prop}
\begin{equation}
	\psi_n(q) - q^{\delta_{2m}} \psi_{n-2m}(q^{4^m}) = \psi_{2m}(q)
	\label{mainrecurrence}
\end{equation}
for all $2 \leq 2m < n$.
\end{prop}

Holding $m$ fixed in \eqref{mainrecurrence} and letting $n \to \infty$ gives a generalization of \eqref{functional}.

\begin{coro}
\[
	f_{\{\delta_n\}}(q) - q^{\delta_{2m}}f_{\{\delta_n\}}(q^{4^m}) = \psi_{2m}(q)
\]
for all $m \geq 1$.
\end{coro}

\begin{table}[htb]
	\centering
	\begin{tabular}{r|l||r|l}
		$n$ & $a(n;z)$ & $n$ & $a(n;z)$ \\ \hline
		$1$ & $1$ & $17$ & $1 + z + z^2 + z^4 + z^8$ \\
		$2$ & $1$ & $18$ & $1 + z^2 + z^4 + z^8$ \\
		$3$ & $1+z$ & $19$ & $1 + z + z^3 + z^4 + z^5 + z^8 + z^9$ \\
		$4$ & $1$ & $20$ & $1 + z^4 + z^8$ \\
		$5$ & $1 + z + z^2$ & $21$ & $1 + z + z^2 + z^5 + z^6 + z^8 + z^9 + z^{10}$ \\
		$6$ & $1 + z^2$ & $22$ & $1 + z^2 + z^6 + z^8 + z^{10}$ \\
		$7$ & $1 + z + z^3$ & $23$ & $1 + z + z^3 + z^7 + z^8 + z^9 + z^{11}$ \\
		$8$ & $1$ & $24$ & $1 + z^8$ \\
		$9$ & $1 + z + z^2 + z^4$ & $25$ & $1 + z + z^2 + z^4 + z^9 +  z^{10} + z^{12}$ \\
		$10$ & $1 + z^2 + z^4$ & $26$ & $1 + z^2 + z^4 + z^{10} + z^{12}$ \\
		$11$ & $1 + z + z^3 + z^4 + z^5$ & $27$ & $1 + z + z^3 + z^4 + z^5 + z^{11} + z^{12} + z^{13}$ \\
		$12$ & $1 + z^4$ & $28$ & $1 + z^4 + z^{12}$ \\
		$13$ & $1 + z + z^2 + z^5 + z^6$ & $29$ & $1 + z + z^2 + z^5 + z^6 + z^{13} + z^{14}$ \\
		$14$ & $1 + z^2 + z^6$ & $30$ & $1 + z^2 + z^6 + z^{14}$ \\
		$15$ & $1 + z + z^3 + z^7$ & $31$ & $1 + z + z^3 + z^7 + z^{15}$ \\
		$16$ & $1$ & $32$ & $1$
	\end{tabular}
	\vspace{2mm}
	\caption{The polynomials $a(n;z)$ for $n = 1,2,\ldots,32$.}
	\label{seqlist}
\end{table}

\section*{Acknowledgements}

The author would like to thank Dante Manna for a suggestion which streamlined the proof of Proposition \ref{unitcircprop} and Karl Dilcher for proof reading and valuable criticism.

\bibliographystyle{amsplain}
\bibliography{biblio}

\end{document}